\documentclass{article}

\usepackage{url}
\usepackage[legalpaper, margin=1.2in]{geometry}
\usepackage{verbatim,amsmath,algorithm,amssymb,amsthm,algorithmic}
\newtheorem{assumption}{Assumption}
\newtheorem{remark}{Remark}
\newtheorem{definition}{Definition}
\newtheorem{lemma}{Lemma}
\newtheorem{theorem}{Theorem}
\usepackage{tikz}
\usepackage{lipsum}
\newtheorem*{lemma*}{Lemma}
\usetikzlibrary{positioning}
\usetikzlibrary{arrows,automata}
\newcommand{\xkh}{x_k^h}

\newcommand{\skh}{s_k^h}

\begin{document}
\title{On high-order multilevel optimization strategies\thanks{This work was funded by TOTAL.}}

\renewcommand{\thefootnote}{\fnsymbol{footnote}}
\footnotetext[1]{INPT-IRIT, University of Toulouse and ENSEEIHT, 2 Rue Camichel, BP~7122, F-31071 Toulouse Cedex~7, France (serge.gratton@enseeiht.fr, elisa.riccietti@enseeiht.fr)}
\author{Henri Calandra,\footnote{TOTAL, Centre Scientifique et Technique Jean F\'eger, avenue de Larribau F-64000 Pau, France   (henri.calandra@total.com)} Serge Gratton\footnotemark[1], Elisa Riccietti\footnotemark[1], Xavier Vasseur\footnote{ISAE-SUPAERO, University of Toulouse, 10, avenue Edouard Belin, BP~ 54032, F-31055 Toulouse Cedex 4, France 	(xavier.vasseur@isae-supaero.fr).}}

\maketitle

\begin{abstract} 
  We propose a new family of multilevel methods for unconstrained minimization. The resulting strategies are multilevel extensions of high-order optimization methods based on $q$-order Taylor models (with $q\geq 1$) that have been recently proposed in the literature. The use of high-order models, while decreasing the worst-case complexity bound, makes these methods computationally more expensive. Hence, to counteract this effect, we propose a multilevel strategy that exploits a hierarchy of problems of decreasing dimension, still approximating the original one, to reduce the global cost of the step computation. A theoretical analysis of the family of methods is proposed. Specifically, local and global convergence results are proved and a complexity bound to reach first order stationary points is also derived.  A multilevel version of the well known adaptive method based on cubic regularization (ARC, corresponding to $q=2$ in our setting) has been implemented. Numerical experiments clearly highlight the relevance of the new multilevel approach leading to considerable computational savings in terms of floating point operations compared to the classical one-level strategy. 
  \end{abstract}

\section{Introduction}
We propose a new family of high-order multilevel optimization methods for unconstrained minimization. Exploiting ideas stemming from multilevel methods allows us to reduce the cost of the step computation, which represents the major cost per iteration of the standard  single level procedures. We have been mainly inspired by two driving ideas: the use of high-order models in optimization as introduced in \cite{Birgin2017}, and the multilevel recursive strategy proposed in \cite{rmtr}. 

When solving unconstrained minimization problems, quadratic models are widely used. These are usually regularized by a quadratic term. For example, trust-region methods have been widely studied and used to globalize Newton-like iterations 
\cite{cgt,nw}. Lately in the literature, a different option has received a growing attention: the use of a cubic overestimator of the objective function
as a regularization technique for the computation of the step from one iterate to the next, giving rise to quadratic models with cubic regularization. This idea first appeared in \cite{griewank} and then was reconsidered in \cite{nesterov}, where the authors proved that the method has a better worst-case complexity bound compared to standard trust-region methods. Later, in \cite{arc,arc2}, an adaptive variant of the method has been proposed, based on a dynamical choice of the regularization parameters and on an approximate solution of the subproblems. The resulting method is known as adaptive method based on cubic regularization (ARC) and is shown to preserve the attractive global complexity bound established in \cite{nesterov}. 
In recent years the method has attracted further interest, see for example \cite{cartis2010complexity,toint2013nonlinear,yuan2015recent}. 

In recent publications also methods of higher order start to gain interest, see for example \cite{Birgin2017,wang2018tensor}. In \cite{Birgin2017} in particular, it has been observed that the good complexity bound of ARC can be made even lower, if one is willing to use higher-order derivatives.  In specific applications this computation is indeed feasible, for example when considering partially separable functions \cite{separable}.  
The authors in \cite{Birgin2017} present a family of methods that generalizes ARC, and that uses high-order regularized models. Specifically, they are based on models of order $q\geq 1$, regularized by a term of order $q+1$. ARC belongs to this family and corresponds to the choice $q=2$.  The authors in \cite{Birgin2017} propose a unifying framework to describe the theoretical properties of the methods in this class. It is proved that the method based on the $q$-th order model requires at most $O\left(\epsilon^{-\frac{q+1}{q}}\right)$ function evaluations to find a first-order critical point, where $\epsilon$ denotes the absolute accuracy level.

However, the use of higher-models come along with higher computational costs. The main cost per iteration of the methods described in \cite{Birgin2017} is represented by the step computation through the model minimization. This cost is proportional to the dimension of the problem, it can therefore be significant for large-scale problems. For second-order models this issue has been faced for example in \cite{rmtr}, where the authors exploit ideas coming from multigrid \cite{book_mg} to reduce the cost of the minimization.
Indeed, the idea of making use of more grids to solve a large-scale problem has been extended also to optimization, see for example \cite{rmtr,mg1,mg2,mg3,mg4,mg5,mg6}. These methods share with classical multigrid methods the idea of exploiting a hierarchy of problems (in this case a sequence of nonlinear functions) defined on lower dimensional spaces, approximating the original objective function $f$. The simplified expressions of the objective function are used to build models that are cheaper to minimize, and are used to define the step. 
Specifically, in \cite{rmtr}, the authors present an extension of classical multigrid methods for nonlinear optimization problems, \cite{bran:77,brli:11} or \cite[Ch. 3]{briggs}, to a class of multilevel trust-region based optimization algorithms.

\paragraph{Our contributions}
\vspace{10pt}
Inspired by the ideas presented in \cite{Birgin2017,rmtr}, we propose a family of multilevel optimization methods using high-order regularized models that generalizes the methods proposed in both papers. The aim is to decrease the computational cost of the methods in \cite{Birgin2017}, extending the ideas in \cite{rmtr} to higher-order models. We also develop a theoretical analysis for the resulting family of methods. The main theoretical results are provided in Theorems \ref{teo_glob}, \ref{teo_compl} and \ref{teo_loc}, respectively. In these theorems we successively prove the global convergence property of the methods,  then evaluate a worst-case complexity bound to reach a first-order critical point and finally provide local convergence rates. The global convergence analysis generalizes the results in \cite{Birgin2017} and appears as much simpler than that in \cite{rmtr}. Moreover we establish local convergence results  towards second-order stationary points, that are not present neither in \cite{Birgin2017} nor in \cite{rmtr}.  These results not only generalize those in \cite{yue2018quadratic}, that are valid only for $q=2$, but also apply to the one level methods in \cite{Birgin2017}. From a practical point of view, we implemented the method of the family corresponding to $q=2$. This represents a multilevel version of ARC method.

 To the best of our knowledge, this is the first time that multilevel optimization strategies, based on models of generic order $q\geq 1$, are proposed, and that a unifying framework is introduced to study their convergence. Moreover, multilevel versions of ARC have never been analysed or tested numerically before.  

The manuscript is organized as follows.  In Section \ref{sec_method}, we briefly introduce the family of optimization methods using high-order regularized models considered in \cite{Birgin2017}. Section \ref{sec_multilevel} and Section \ref{sec_conv} represent our main contribution. We introduce in Section \ref{sec_multilevel} the multilevel extensions of the methods presented in Section \ref{sec_method}, and we provide  a theoretical analysis in Section \ref{sec_conv}. Specifically, we focus on global convergence in Section   \ref{sec_glob}, worst-case complexity in Section \ref{sec_compl} and local convergence in Section \ref{sec_loc}.  In Section \ref{sec_numerical_results} we then present results related to numerical experiments performed with the multilevel method corresponding to $q=2$. 
Finally, conclusions are drawn in Section \ref{sec_conclusion}.

\section{High-order iterative optimization methods}\label{sec_method}
Let $q\geq 1$ be an integer. Let us consider a minimization problem of the form: 
\begin{equation}\label{nls}
\min_{x \in  \mathbb{R}^n} f(x)
\end{equation}
with $f:\mathbb{R}^n\rightarrow\mathbb{R}$ a bounded below and $q$-times continuously differentiable function, called the objective function. 

Classical iterative optimization methods for unconstrained minimization are based on the use of a model to approximate the objective function at each iteration. In this section, we describe the iterative optimization methods using high-order models presented in \cite{Birgin2017}.

\subsection{Model definition and step acceptance}
At each iteration $k$, given the current iterate $x_k$, the objective function is approximated by the Taylor series $T_{q,k}$ of $f(x_k+s)$ (with $s \in \mathbb{R}^n$) truncated at order $q$. The Taylor model of order $q$ denoted as $m_{q,k}$ is then defined as:
\begin{equation}\label{taylor_model}
m_{q,k}(x_k,s)=T_{q,k}(x_k,s).
\end{equation} 
 

 A step $s_k$ is then found minimizing (possibly approximately) the regularized model
	\begin{equation}\label{reg_taylor_model}
	T_{q,k}(x_k,s)+\frac{\lambda_k}{q+1}\|s\|^{q+1},
	\end{equation}
	where $\lambda_k$ is a positive value called regularization parameter. The step $s_k$ is used to define a trial point i.e. $x_{k+1}=x_k+s_k$. At each iteration, it has to be decided whether to accept the step or not. This decision is based on the accordance between the decrease in the function and in the model. More precisely, at each iteration both the decrease achieved in the model, that we call \textit{predicted reduction},  $pred=m_{q,k}(x_k)-m_{q,k}(x_k,s_k)$, and that achieved in the objective function, that we call \textit{actual reduction},  $ared=f(x_k)-f(x_k+s_k)$, are computed. The step acceptance is then based on the ratio:
\begin{equation}\label{rho}
\rho_k=\frac{ared}{pred}=\frac{f(x_k)-f(x_{k}+s_k)}{m_{q,k}(x_k)-m_{q,k}(x_k,s_k)}.
\end{equation} 
If the model is accurate, $\rho_k$ will be close to one. Then, the step $s_k$ is accepted if $\rho_k$ is larger than or equal to a chosen threshold $\eta_1\in(0,1)$ and is rejected otherwise. In the first case, the step is said to be \textit{successful}, and otherwise the step is \textit{unsuccessful}.

After the step acceptance, the regularization parameter is updated for the next iteration. The update is still based on the ratio \eqref{rho}. If the step is successful, the regularization parameter is decreased, otherwise it is increased. The whole procedure is stopped when a minimizer of $f$ is reached. Usually, the stopping criterion is based on the norm of the gradient, i.e. given an absolute accuracy level $\epsilon>0$ the iterations are stopped as soon as $\|\nabla_x f(x_k)\|<\epsilon$. The whole procedure is sketched in Algorithm \ref{algo0}.

\begin{algorithm}
\caption{AR$\mathbf{q}(x_0, \lambda_0, \epsilon)$ (Adaptive Regularization method of order $q$)}
\label{algo0}
\begin{algorithmic}[1]
	\STATE{Given $0<\eta_1\leq\eta_2<1$,  $0<\gamma_2\leq\gamma_1< 1<\gamma_3$, $\lambda_{\min}>0$.}
\STATE{{\bf Input:} $x_0 \in \mathbb{R}^n$, $\lambda_0 > \lambda_{\min} $, $\epsilon > 0$.}
\STATE{$k = 0$}
\WHILE{$\|\nabla_x f(x_k)\|> \epsilon$}
\STATE{$\bullet$ {\bf Initialization:} Define the model $m_{q,k}$ as in \eqref{taylor_model}.}
\STATE{$\bullet$ {\bf Model minimization:} Find a step $s_k$ that sufficiently reduces the model. }
\STATE{$\bullet$ {\bf Acceptance of the trial point:}\label{step_acceptance0} Compute $\rho_k=\displaystyle \frac{f(x_k)-f(x_k+s_k)}{m_{q,k}(x_k)-m_{q,k}(x_k,s_k)}$.
\IF{$\rho_k\geq \eta_1$} \STATE{$x_{k+1}=x_k+s_k$} \ELSE \STATE{$x_{k+1}=x_k$.}\ENDIF}
\STATE{$\bullet$ {\bf Regularization parameter update:} }
\IF{$\rho_k\geq \eta_1$} \STATE{
		$$\lambda_{k+1} =
		\bigg \{
		\begin{array}{ll}
		\max\{\lambda_{\min},\gamma_2\lambda_k\},  & \text{ if }\rho_k\geq \eta_2, \\
		\max\{\lambda_{\min},\gamma_1\lambda_k\},  & \text{ if } \rho_k< \eta_2,\\
		\end{array}
		$$}\ELSE \STATE{ 
		$\lambda_{k+1}= \gamma_3\lambda_k$.}\ENDIF
\STATE{$k = k + 1$}
\ENDWHILE
\end{algorithmic}
\end{algorithm}

\subsection{Minimization of the model}
The main computational work per iteration in this kind of methods is represented by the minimization of the regularized model \eqref{reg_taylor_model}. This is the most expensive task, and the cost naturally depends on the dimension of the problem. However, from the convergence theory of such methods, it is well known that it is not necessary to minimize the model exactly to get a globally convergent method. 

A well-known possibility is to minimize the model until the Cauchy decrease is achieved, i.e. until a fraction of the decrease provided by the Cauchy step  (the step that minimizes the model in the direction of the negative gradient) is obtained. 
In \cite{Birgin2017} the authors consider a different stopping criterion for the inner iterations, the one originally proposed in \cite{arc,arc2}, which has the advantage of allowing for simpler convergence proofs. The inner iterations are stopped as soon as the norm of the gradient of the regularized model becomes lower or equal than a multiple of the power $q$ of the norm of the step $s_k$:
\begin{equation}\label{stopping_inner}
\|\nabla_s m_{q,k}(x_k,s_k)+\lambda_k\|s_k\|^{q-1} s_k\|\leq \theta \|s_k\|^q,
\end{equation}
for a chosen constant $\theta>0$.

For very large-scale problems however, even an approximate minimization of  \eqref{reg_taylor_model} may be really costly. Then, in the next section we propose  multilevel variants of the procedures, that rely on simplified models of the objective function, cheaper to optimize, allowing to reduce the global cost of the optimization procedure.

\section{Multilevel optimization methods}\label{sec_multilevel}

We describe the multilevel extension of the family of methods presented in Section \ref{sec_method}. The procedures are inspired by the multilevel trust-region approach presented in \cite{rmtr}, where only second-order models with quadratic regularization have been considered. Here, we generalize this approach by allowing also higher-order models, i.e. $q>2$.

\subsection{Preliminaries and notations}
In standard optimization methods the minimization of \eqref{reg_taylor_model} represents the major cost per iteration, which crucially depends on the dimension $n$ of the problem.
When $n$ is large, the solution cost is therefore often significant.
We want to reduce this cost by exploiting
the knowledge of alternative simplified expressions of the objective function. More specifically, we assume that we know a collection of functions $\{f_l\}_{l=1}^{l_{\max}}$ 
such that each $f_l$ is a $q$-times continuously differentiable function from $\mathbb{R}^{n_l}\rightarrow\mathbb{R}$  and  $f^{l_{\max}}(x)=f(x)$ for all $x\in\mathbb{R}^n$. We will also assume that, for each $l = 2,\dots,{l_{\max}}$, $f_l$ is more
costly to minimize than $f_{l-1}$. This is the typical scenario when the problem arises from the discretization of an infinite dimensional problem and $f_l$ represent increasingly finer discretizations. In this case,
$n_l\geq n_{l-1}$ for all $l$, but of course this is not the only possible application. As we do not assume the hierarchy to come from a discretization process, we do not use the terminology typically used in the field of multigrid methods. We will then use 'levels' rather than 'grids'. 

The methods we  propose are recursive procedures, so it suffices to describe the two-level case. Then, for sake of simplicity, from now on, we will assume that we have just two approximations to our objective $f$ at disposal. This amounts to consider $l_{\max}=2$.

For ease of notation, we will denote by $f^h:\mathcal{D}^h\subseteq\mathbb{R}^{n_h}\rightarrow\mathbb{R}$ the approximation at the highest level ($f^h(x)=f^{l_{\max}}(x)$ in the notation previously used) and by $f^H:\mathcal{D}^H\subseteq\mathbb{R}^{n_H}\rightarrow\mathbb{R}$ the other approximation available, that is cheaper to optimize. The quantities on the highest level will be denoted by a superscript $h$, whereas the quantities on the lower level will be denoted by a superscript $H$. Let $\xkh$ denote the $k$-th iteration at the highest level.   In the following $\|\cdot\|$ will denote the Euclidean norm. We will use the same notation for all the spaces we will consider, the space on which the norm is defined will be clear by the context.

\begin{remark}
	 To deal with high order derivatives, and to properly handle the concept of coherence between lower and higher level model,  we will need to use a tensor notation, that we introduce here for convenience of the reader, see \cite{tensor}. We first consider a tensor of order three, and then extend the definition to a tensor of order $i\in\mathbb{N}$. 
		 \begin{definition}
		 	Let $T\in\mathbb{R}^{n\times n\times n}$, and $u,v,w\in\mathbb{R}^{n}$. Then $T(u,v,w)\in\mathbb{R}$, $T(u,v)\in\mathbb{R}^n$  and
		 	\begin{align*}
		 	T(u,v,w)&=\sum_{i=1}^{n}\sum_{j=1}^{n}\sum_{k=1}^{n}T(i,j,k)u(i)v(j)w(k),\\
		 	T(v,w)(i)&=\sum_{j=1}^{n}\sum_{k=1}^{n}T(i,j,k)v(j)w(k), \quad i=1,\dots,n.
		 	\end{align*}
		 	\end{definition}
		 	
		 	\begin{definition}\label{def_tensor}
		 		Let $i\in\mathbb{N}$ and $T\in\mathbb{R}^{n^i}$, and $u_1,\dots,u_i\in\mathbb{R}^{n}$. Then $T(u_1,\dots,u_i)\in\mathbb{R}$,
		 		$T(u_1,\dots,u_{i-1})\in\mathbb{R}^n$  and
		 		\begin{align*}
		 		T(u_1,\dots,u_i)&=\sum_{j_1=1}^{n}\dots\sum_{j_i=1}^{n}T(j_1,\dots,j_i)u_1(j_1)\dots u_i(j_i),\\
		 		T(u_1,\dots,u_{i-1})(j_1)&=\sum_{j_2=1}^{n}\dots\sum_{j_i=1}^{n}T(j_1,\dots,j_i)u_i(j_2),\dots u_{i-1}(j_i), \quad j_1=1,\dots,n.
		 		\end{align*}
		 	\end{definition}
		 	
		 \end{remark}

\subsection{Construction of the lower level model}
The main idea is to use $f^H$ to construct, in the neighbourhood of the current iterate, an alternative model $m_{q,k}^H$ to the Taylor model $m_{q,k}^h$ in \eqref{taylor_model} for $f^h=f$ \cite{rmtr}.
The alternative model $m_{q,k}^H$ should be cheaper to optimize than $m_{q,k}^h$, and will be used, whenever suitable, to define the
step. Of course, for $f^H$ to be useful at all in minimizing $f^h$, there should be some
relation between the variables of these two functions. We henceforth assume the following. 

\begin{assumption}\label{ass_R}
	Let us assume that there exist two full-rank linear operators $R:\mathbb{R}^{n_h}\rightarrow\mathbb{R}^{n_H}$ and $P:\mathbb{R}^{n_H}\rightarrow\mathbb{R}^{n_h}$ such that 
	$P=\alpha R^T$, for a fixed scalar $\alpha>0$. Let us assume also that it exists $\kappa_R>0$ such  that $\max\{\|R\|,\|P\|\}\leq \kappa_R$, where $\|\cdot\|$ denotes the matrix norm induced by the Euclidean norm at the fine level.
\end{assumption} 

In the following, we can assume $\alpha=1$, without loss of generality,   as the problem can be easily scaled to handle the case $\alpha\neq 1$.

At each iteration $k$ at highest level we set $x_{0,k}^H=R~\xkh$, i.e.  the initial iterate at the lower level is set as the projection of the current iterate, and we define the lower level model $m_{q,k}^H$ as a modification of the coarse function $f^H$. Given $q$, $f^H$ is modified adding $q$ correction terms, to enforce the following relation:
\begin{equation}\label{def_mathcalR}
\nabla^i_s m_{q,k}^H(x_{0,k}^H,\underbrace{s^H,\dots,s^H}_{i\,\mathrm{times}})=[\mathcal{R}(\nabla^i_x f^h(x_k^h))](\underbrace{s^H,\dots,s^H}_{i\,\mathrm{times}}), \quad i=1,\dots,q,
\end{equation}
where $\mathcal{R}(\nabla^i_x f^h(x_k^h))$ is such that for all $i=1,\dots,q$ and $s_{1}^H, \dots, s_{i}^H\in\mathbb{R}^{n_H}$
\begin{align}
	[\mathcal{R}(\nabla^i_x f^h(x_k^h))](s_{1}^H,\dots,s_{i}^H)&:=\nabla^i_x f^h(\xkh,Ps_{1}^H,\dots,Ps_{i}^H),\label{def1}\\
	\langle [\mathcal{R}(\nabla^i_x f^h(x_k^h))](s_{1}^H,\dots,s_{i-1}^H), s_{i}^H\rangle& := \langle \nabla^i_x f^h(\xkh,Ps_{1}^H,\dots,Ps_{i-1}^H),P s_{i}^H\rangle\label{def2},
\end{align}
 where $\nabla^i_x f^h, \nabla^i_s m_{q,k}^H$ denote the $i$-th order tensor of $f^h$ and $m_{q,k}^H$ respectively, $\langle\cdot,\cdot\rangle$ denotes the scalar product, and for generic $g:\mathbb{R}^n\rightarrow\mathbb{R}$ and $s_1,\dots,s_i\in\mathbb{R}^n$, $\nabla ^i g(x,s_1,\dots,s_i)$ is the same as $\nabla ^i g(x)(s_1,\dots,s_i)$, which is given in Definition \ref{def_tensor}. 

For instance, if 
$q=2$, relation \eqref{def_mathcalR} simply becomes:
\begin{equation*}
\nabla_s m_{q,k}^H(x_{0,k}^H)^Ts^H=(R~\nabla_x f^h(\xkh))^Ts^H, \quad(s^H)^T\nabla_x^2 m_{q,k}^H(x_{0,k}^H)s^H=(s^H)^TR~\nabla_x^2 f^h(\xkh)~Ps^H.
\end{equation*}
Relation \eqref{def_mathcalR} crucially ensures that  the behaviours of $f^h$ and $m_{q,k}^H$ are coherent up to order $q$ in a
neighbourhood of $\xkh$ and $x_{0,k}^H$. To achieve \eqref{def_mathcalR}, we define $m_{q,k}^H$ as
\begin{equation}\label{lower_level_model}
m_{q,k}^H(x_{0,k}^H, s^H)=f^H(x_{0,k}^H+s^H)+\sum_{i=1}^{q}\frac{1}{i!}[\mathcal{R}( \nabla^i_x f^h(\xkh)) -\nabla^i_x f^H(x_{0,k}^H)](\underbrace{s^H,\dots,s^H}_{i\,\mathrm{times}}),
\end{equation}
with $\mathcal{R}( \nabla^i_x f^h(\xkh))$ defined in \eqref{def1}-\eqref{def2}.
When $q=2$ this is simply:
\begin{align*}
m_{2,k}^H(x_{0,k}^H, s^H)=&f^H(x_{0,k}^H+s^H)+(R \nabla_x f^h(\xkh)-\nabla_x f^H(x_{0,k}^H))^Ts^H\nonumber\\
+&\frac{1}{2}(s^H)^T (R \nabla_x^2 f^h(\xkh)P-\nabla_x^2 f^H(x_{0,k}^H))s^H.
\end{align*}
\subsection{Step computation and step acceptance}
At each generic iteration $k$ of our method, a step $s^h_k$ has to be computed to define the new iterate. Then, one has the choice between the Taylor model \eqref{taylor_model} and a lower level model  \eqref{lower_level_model}.

Obviously, it is not always possible to use the lower level model. For example, it may happen that $\nabla_x f^h(\xkh)$ lies in the nullspace of $R$ and thus that $R \nabla_x f^h(\xkh)$ is zero while $ \nabla_x f^h(\xkh)$ 
is not. In this case, the current iterate appears to be first-order critical for $m_{q,k}^H$  while it is not for $f^h$. Using the model $m_{q,k}^H$ is hence potentially useful
only if $\|\nabla_s m_{q,k}^H(x_{0,k}^H)\| = \|R \nabla_x f^h(\xkh)\|$ is large enough compared to $\|\nabla_x f^h(\xkh)\|$ \cite{rmtr}. We therefore restrict
the use of the model $m_{q,k}^H$ to iterations where
\begin{equation}\label{go_down_condition}
\|R\nabla_x f^h(\xkh)\|\geq \kappa_H \|\nabla_x f^h(\xkh)\| \quad \text{ and } \quad \|R \nabla_x f^h(\xkh)\|> \epsilon_H,
\end{equation}
for some constant $\kappa_H\in(0,\min\{1,\|R\|\})$ and where $\epsilon_H\in(0,1)$  is a measure of
the first-order criticality for $m_{q,k}^H$ that is judged sufficient at level $H$ \cite{rmtr}. Note that,
given $\nabla_x f^h(\xkh)$ and $R$, this condition is easy to check before even attempting to compute
a step at a lower level.

If the Taylor model is chosen, then we just compute a step as in standard methods, minimizing (possibly approximately) the corresponding regularized model \eqref{reg_taylor_model}. If the lower level model is chosen, we then minimize the following regularized model:
\begin{equation}\label{reg_lower_model}
m_{q,k}^H(x_{0,k}^H, s^H)+\frac{\lambda_k}{q+1}\|s^H\|^{q+1}
\end{equation} 
(possibly approximately) and obtain a point $x_{*,k}^H$ such that (if the minimization is successful) the value of the regularized model has been reduced, and a step $s_k^H=x_{*,k}^H-x_{0,k}^H$ (note that the iteration indices always refer to the highest level, we are not indexing the iterations on the lower level for the minimization of the lower level model). This step has to be prolongated back on the fine level, i.e. we define $s_k^h=Ps_k^H$.

Then, $m_{q,k}^h$ will be defined as:
\begin{equation}\label{choice_model}
m_{q,k}^h(x_k^h,s_k^h)=\begin{cases} T_{q,k}^h(x_k^h,s_k^h) & \text{(Taylor model)},\\
	m_{q,k}^H(Rx_k^h,s_k^H), \;s_k^h=Ps_k^H &\text{(lower level model)}.
	\end{cases}
\end{equation} 

In both cases, after the step is found, we have to decide whether to accept it or not. The step acceptance is based on the ratio:
\begin{equation*}
\rho_k=\frac{f^h(\xkh)-f^h(\xkh+s_k^h)}{m^h_{q,k}(\xkh)-m^h_{q,k}(\xkh,\skh)},
\end{equation*}
where we remind that, from \eqref{choice_model}, the denominator is defined as:
\begin{equation}\label{pred}
m^h_{q,k}(\xkh)-m^h_{q,k}(\xkh,\skh)=	\begin{cases}T_{q,k}^h(\xkh)-T_{q,k}^h(\xkh+s_k^h),& \text{(Taylor model)},\\
	
	m_{q,k}^H(R\xkh)-m_{q,k}^H(R\xkh, s_k^H),& \text{(lower level model)}.
	\end{cases} 
\end{equation}
As in the standard form of the methods, the step is accepted if it provides a sufficient decrease in the function, i.e. if given $\eta_1>0$, $\rho_k\geq\eta_1$. The regularization parameter is also updated as in Algorithm \ref{algo0}. We sketch the whole procedure in Algorithm \ref{algo}.

\begin{algorithm}{}
	\begin{algorithmic}[1]
    \caption{MAR$\mathbf{q}(l,f^l, x_{0}^l, \lambda_{0}^l, \epsilon^l)$ (Multilevel Adaptive Regularization method of order $q$)}\label{algo} 
    \STATE{{\bf Input:}  $l \in \mathbb{N}$ (index of the current level, $1 \le l \le l_{\max}$, $l_{\max}$ being the highest level), $f^l:\mathbb{R}^{n_l}\rightarrow\mathbb{R}$ function to be optimized ($f^{l_{\max}}=f$), $x_{0}^l \in \mathbb{R}^{n_l}$, $\lambda_{0}^l > \lambda_{\min} $, $\epsilon^l > 0$.
    	}
    \STATE{Given $0<\eta_1\leq\eta_2<1$,  $0<\gamma_2\leq\gamma_1< 1<\gamma_3$, $\lambda_{\min}>0$.}
	
	\STATE{$R_l$ denotes the restriction operator from level $l$ to $l-1$, $P_l$ the prolongation operator from level $l-1$ to $l$.}
	\STATE{$k = 0$}
	\WHILE{$\|\nabla_x f^l(x_k^l)\|> \epsilon^l$}
		\STATE{$\bullet$ \label{model_choice}{\bf Model choice:}
		If $l>1$ compute $R_l\nabla_x f^l(x_k^l)$ and check \eqref{go_down_condition}. If $l=1$ or \eqref{go_down_condition} fails, go to Step \ref{step_taylor}. Otherwise, choose to go to Step \ref{step_taylor} or to Step \ref{step_lower}. }
		\STATE{$\bullet$ \label{step_taylor}  {\bf Taylor step computation:}  Define $m_{q,k}^l(x_k^l,s^l)=T_{q,k}^l(x_k^l,s^l)$, the Taylor series of $f^l(x_k^l+s^l)$ truncated at order $q$. Find a step $s_k^l$ that sufficiently reduces $m_{q,k}^l(x_k^l,s^l)+\frac{\lambda_k^l}{q+1}\|s^l\|^{q+1}$. Go to Step \ref{step_acceptance}.}
		\STATE{$\bullet$ \label{step_lower} {\bf Recursive step computation:}
		Define 
		\begin{align*}
		m_{q,k}^{l-1}(R_l~x_k^l,s^{l-1})&=f^{l-1}(R_l~x_k^l,s^{l-1})\\
		&+\sum_{i=1}^{q}\frac{1}{i!}[\mathcal{R}( \nabla^i_x f^l(x_k^l)) -\nabla^i_x f^{l-1}(R_l~x_{k}^l)](\underbrace{s^{l-1},\dots,s^{l-1}}_{i\,\mathrm{times}}).
		\end{align*}
		 Choose $\epsilon^{l-1}$ and call MAR$\mathbf{q}$($l-1$, $m_{q,k}^{l-1}$,$R_l~x_{k}^l$, $\lambda_k^l$, $\epsilon^{l-1}$)
		yielding an approximate solution $x^{l-1}_{*,k}$ of the minimization of $m_{q,k}^{l-1}$. Define  $s_{k}^l=P_l~(x^{l-1}_{*,k}-R_l~x_k^l)$ and $m_{q,k}^l(x_k^l,s^l)=m_{q,k}^{l-1}(R_l~x_k^l,s^{l-1})$ for all $s^l=Ps^{l-1}$.}
		\STATE{$\bullet$ \label{step_acceptance} {\bf Acceptance of the trial point:} Compute $\rho_k^{l}=\displaystyle \frac{f^l(x_k^l)-f^l(x_k^l+s_k^l)}{m_{q,k}^l(x_k^l)-m_{q,k}^l(x_k^l,s_k^l)}.$ }
		\IF{\label{step3a} $\rho_k^l\geq \eta_1$} \STATE{ $x_{k+1}^l=x_k^l+s_k^{l}$} \ELSE \STATE{\label{step3b}  $x_{k+1}^l=x_k^l$.}\ENDIF
		\STATE{$\bullet$ {\bf Regularization parameter update:} }
		\IF{$\rho_k^l\geq \eta_1$ } \STATE{ 
			$$\lambda_{k+1}^l =
			\bigg \{
			\begin{array}{ll}
			\max\{\lambda_{\min},\gamma_2\lambda_k^l\},  & \text{ if }\rho_k^l\geq \eta_2, \\
			\max\{\lambda_{\min},\gamma_1\lambda_k^l\},  & \text{ if } \rho_k^l< \eta_2\\
			\end{array}
			$$}
			\ELSE\STATE{ 
			$\lambda_{k+1}^l= \gamma_3\lambda_k^l$.}\ENDIF
	    \STATE{$k = k + 1$}
		\ENDWHILE
		\end{algorithmic}
\end{algorithm}

 Some comments are necessary to explain Step \ref{model_choice} in Algorithm \ref{algo}. The generic framework sketched in Algorithm \ref{algo} comprises different possible methods. Specifically, one of the flexible features (inherited by the method in \cite{rmtr}) is that, to ensure convergence, the
	minimization at lower levels can be stopped after the first successful
	iteration, as we will see in the next section. This therefore opens the possibility to consider both fixed form recursion patterns and free
	form ones. A free form pattern is obtained when Algorithm \ref{algo} is run carrying the minimization 
	at each level out, until the norm of the gradient becomes small enough. The actual
	recursion pattern is then uniquely determined by the progress of minimization at
	each level and may be difficult to forecast. By contrast, the fixed form recursion
	patterns are obtained by specifying a maximum number of successful iterations at
	each level, a technique directly inspired from the definitions of V- and W-cycles in
	multigrid algorithms \cite{book_mg}.

\section{Convergence theory}\label{sec_conv}
In this section, we provide a theoretical analysis of the proposed family of multilevel methods. Inspired by the convergence theory reported in \cite{Birgin2017}, we prove global convergence of the proposed methods to first-order critical points and we provide a worst-case complexity bound to reach such a point, generalizing the theory proposed in \cite{Birgin2017,rmtr}. At the same time the proposed analysis also appears as simpler than that in \cite{rmtr}, since the regularization parameter $\lambda_k$ is directly updated, rather than the trust-region radius, and since we use the stopping criterion \eqref{stopping_inner} as in \cite{Birgin2017}.  The use of this criterion allows for simpler convergence proofs and enables us to concentrate on the multilevel algorithm, that is the main contribution of the paper. Moreover, we also propose local convergence results, which also apply to the methods in \cite{Birgin2017}, and that extend those in \cite{yue2018quadratic} to higher-order models. 

Note that, as the methods are recursive, we can restrict the analysis to the two-level case.  
For the analysis we need the following regularity assumptions as  in \cite{Birgin2017}. 
\begin{assumption}\label{hp_lip_grad}
	Let $f^h$ and $f^H$ be $q$-times continuously differentiable and bounded below functions.
	Let us assume that the $q$-th derivative tensors of $f^h$ and $f^H$  are Lipschitz continuous, i.e. that there exist constants $L_h,L_H$ such that 
	\begin{align*}
	\|\nabla^q f^h(x)-\nabla^q f^h(y)\|_T &\leq (q-1)!\, L_h\, \|x-y\|\quad \text{ for all} \quad x,y\in\mathcal{D}^h,\\
	\|\nabla^q f^H(x)-\nabla^q f^H(y)\|_T &\leq (q-1)!\, L_H\, \|x-y\|\quad \text{ for all} \quad x,y\in\mathcal{D}^H,
	\end{align*}
	where $\|\cdot\|_T$ is the tensor norm recursively induced by the Euclidean norm on the space of $q$-th order tensors, which for a tensor $H$ of order $q$ is given by 
	\begin{equation*}
	\|H\|_T\overset{\mathrm{def}}{=}\max_{\|u_1\|=\dots=\|u_q\|=1}\lvert H(u_1,\dots,u_q)\rvert
	\end{equation*}
	where the action of $H$ on $(u_1,\dots,u_q)$ is given in Definition \eqref{def_tensor}. 
\end{assumption}

We remind three useful relations, following from Taylor's  theorem, see for example relations (2.3) and (2.4) in \cite{Birgin2017}.

\begin{lemma}
	 Let $g:\mathbb{R}^n\rightarrow\mathbb{R}$ be a $q$-times continuously differentiable function with  Lipschitz continuous  $q$-th order tensor, with $L$ the corresponding Lipschitz constant. Given its Taylor series $T_q(x,s)$ truncated at order $q$, it holds:
\begin{align}
g(x+s)=&T_q(x,s)+\frac{1}{(q-1)!}\int_{0}^{1}(1-\xi)^{q-1}[\nabla^q g(x+\xi s)-\nabla^q g(x)](\overbrace{s,\dots,s}^{q\,\mathrm{times}})\,d\xi,\label{taylor1}\\
&\lvert g(x+s)- T_q(x,s)\rvert \leq\frac{L}{q}\|s\|^{q+1},\label{taylor2}\\
&\|\nabla g(x+s)-\nabla_s T_q(x,s)\|\leq L\|s\|^{q}.\label{taylor3}
\end{align} 
\end{lemma}

\subsection{Global convergence}\label{sec_glob}
In this section we prove the global convergence property of the method. Our analysis proceeds in three steps. First, we bound the quantity $\lvert1-\rho_k\rvert$ to prove that $\lambda_k$ must be bounded above. Then, we relate the norm of the step and the norm of the gradient. Finally, we use these two ingredients to conclude proving that the norm of the gradient goes to zero. 

\subsubsection{Upper bound for the regularization parameter $\lambda_k$}
At iteration $k$ we either minimize (decrease) the regularized Taylor model \eqref{reg_taylor_model}, or the regularized lower level model \eqref{reg_lower_model}. 
Consequently, it respectively holds:
\begin{subequations}\label{dent}
	\begin{align}
	T_{q,k}^h(x_{k}^h)-T_{q,k}^h(x_{k}^h, s_k^h)&\geq \frac{\lambda_k}{q+1}\|s_k^h\|^{q+1},\label{den2}\\
	m_{q,k}^H(x_{0,k}^H)-m_{q,k}^H(x_{0,k}^H, s_k^H)&\geq \frac{\lambda_k}{q+1}\|s_k^H\|^{q+1}.\label{den}
	\end{align}
\end{subequations}
In both cases, the minimization process is stopped as soon as the stopping condition 
\begin{align}
\|\nabla_s T_{q,k}^h(\xkh,s_k^h)+\lambda_k\|\skh\|^{q-1}\skh\|&\leq \theta \|s_k^h\|^q, \quad \text{ or }\nonumber\\
	\|\nabla_s m_{q,k}^H(x_{0,k}^H,s_k^H)+\lambda_k\|s_k^H\|^{q-1}s_k^H\|&\leq \theta \|s_k^H\|^{q},\label{stopping}
\end{align}
\noindent for $\theta>0$ is satisfied, respectively.  In both cases we are sure that it will exist a point that satisfies \eqref{stopping}, as when the level is selected, a standard one-level optimization method is used, and the analysis in \cite{Birgin2017} applies. 

Let us consider the quantity
\begin{equation}\label{ratio}
\lvert 1-\rho_k\rvert =\Bigg\lvert 1-\frac{f^h(\xkh)-f^h(\xkh+s_k^h)}{m_{q,k}^h(\xkh)-m_{q,k}^h(\xkh,s_k^h)}\Bigg\rvert,
\end{equation}
with the denominator defined in \eqref{pred}. If at step $k$ the Taylor model is chosen, from relation \eqref{taylor2} applied to $f^h$ and \eqref{den2} we obtain the inequality:
\begin{equation*}
\lvert 1-\rho_k\rvert =\Bigg\lvert \frac{f^h(\xkh+s_k^h)-T_{q,k}^h(\xkh,s_k^h)}{m_{q,k}^h(\xkh)-m_{q,k}^h(\xkh,s_k^h)}\Bigg\rvert\leq \frac{L_h(q+1)}{\lambda_k q}.
\end{equation*}
If the lower level model is used, we have
\begin{equation*}
\lvert 1-\rho_k\rvert =\Bigg\lvert \frac{m_{q,k}^H(x_{0,k}^H)-m_{q,k}^H(x_{0,k}^H,s_k^H)-(f^h(\xkh)-f^h(\xkh+s_k^h))}{m_{q,k}^H(x_{0,k}^H)-m_{q,k}^H(x_{0,k}^H,s_k^H)}\Bigg\rvert.
\end{equation*}
Let us consider the numerator in this expression.  
From relations \eqref{lower_level_model} and \eqref{taylor1} applied to $f^H$, using its Taylor series $T_{q,k}^H$, it follows 
\begin{align}
&m_{q,k}^H(x_{0,k}^H)-m_{q,k}^H(x_{0,k}^H,s_k^H) \overset{\eqref{lower_level_model}}{=}\nonumber\\
&T_{q,k}^H(x_{0,k}^H,s_k^H)-f^H(x_{0,k}^H+s_k^H)-\sum_{i=1}^{q}\frac{1}{i!}\left[\mathcal{R}( \nabla^i_x f^h(\xkh)) \right](\overbrace{s_k^H,\dots,s_k^H}^{i\,\mathrm{times}}),\nonumber\\
\overset{\eqref{taylor1}}{=}
&-\frac{1}{(q-1)!}\int_{0}^{1} (1-\xi)^{q-1}\left[\nabla^q f^H(x_{0,k}^H+\xi s_k^H)-\nabla^q f^H(x_{0,k}^H)\right](\overbrace{s_k^H,\dots,s_k^H}^{q\;\mathrm{times}})d\xi\nonumber\\
&-\sum_{i=1}^{q}\frac{1}{i!}\left[\mathcal{R}( \nabla^i_x f^h(\xkh)) \right](\overbrace{s_k^H,\dots,s_k^H}^{i\,\mathrm{times}}).\label{num_low}
\end{align}
Similarly the relation \eqref{taylor1} applied to $f^h$ yields
\begin{align}
f^h(\xkh)&-f^h(\xkh+s_k^h)=f^h(\xkh)-T_{q,k}^h(\xkh,s_k^h)\nonumber\\
&-\frac{1}{(q-1)!}\int_{0}^{1}(1-\xi)^{q-1}[\nabla^qf^h(\xkh+\xi s_k^h)-\nabla^q f^h(\xkh)](\overbrace{s_k^h,\dots,s_k^h}^{q\; \mathrm{times}})d\xi.\label{num_up}
\end{align}
From relation \eqref{def_mathcalR} we can rewrite $f^h(\xkh)-T_{q,k}^h(\xkh,s_k^h)$ as:
\begin{align*}
f^h(\xkh)-T_{q,k}^h(\xkh,s_k^h)=&-\sum_{i=1}^{q}\frac{1}{i!}( \nabla^i_x f^h)(\xkh,\overbrace{Ps_k^H,\dots,Ps_k^H}^{i\; \mathrm{times}})\\
=&-\sum_{i=1}^{q}\frac{1}{i!}\left[\mathcal{R}( \nabla^i_x f^h(\xkh))\right](\overbrace{s_k^H,\dots,s_k^H}^{i\; \mathrm{times}}).
\end{align*}
Then, subtracting \eqref{num_up} from \eqref{num_low}, we obtain
\begin{align*}
&m_{q,k}^H(x_{0,k}^H)-m_{q,k}^H(x_{0,k}^H, s_k^H)-(f^h(\xkh)-f^h(\xkh+s_k^h))=\\
&-\frac{1}{(q-1)!}\int_{0}^{1} (1-\xi)^{q-1}[\nabla^q f^H(x_{0,k}^H+\xi s_k^H)-\nabla^q f^H(x_{0,k}^H)](\overbrace{s_k^H,\dots,s_k^H}^{q\,\mathrm{times}})\,d\xi\\
&+\frac{1}{(q-1)!}\int_{0}^{1} (1-\xi)^{q-1}[\nabla^q f^h(\xkh+\xi s_k^h)-\nabla^q f^h(x_{0,k}^h)](\overbrace{s_k^h,\dots,s_k^h}^{q\,\mathrm{times}})\,d\xi.
\end{align*}
Using Assumption \ref{hp_lip_grad}, we obtain:
\begin{align*}
&\lvert m_{q,k}^H(x_{0,k}^H)-m_{q,k}^H(x_{0,k}^H, s_k^H)-(f^h(\xkh)-f^h(\xkh+s_k^h))\rvert\\
&\leq\frac{1}{(q-1)!}\int_{0}^{1} (1-\xi)^{q-1}\lvert[\nabla^q f^H(x_{0,k}^H+\xi s_k^H)-\nabla^q f^H(x_{0,k}^H)](\overbrace{s_k^H,\dots,s_k^H}^{q\,\mathrm{times}})\rvert \,d\xi\\
&	+\frac{1}{(q-1)!}\int_{0}^{1} (1-\xi)^{q-1}\lvert[\nabla^q f^h(\xkh+\xi s_k^h)-\nabla^q f^h(x_{k}^h)](\overbrace{s_k^h,\dots,s_k^h}^{q\,\mathrm{times}})\rvert\, d\xi\\
&\leq \frac{1}{q!}\|s_k^H\|^q\max_{\xi\in[0,1]}\|\nabla^q f^H(x_k^H+\xi s_k^H)-\nabla^q f^H(x_k^H)\|_T\\
&+\frac{1}{q!}\|s_k^h\|^q\max_{\xi\in[0,1]}\|\nabla^q f^h(\xkh+\xi s_k^h)-\nabla^q f^h(\xkh)\|_T\leq \frac{1}{q}\left(L_H+L_h\kappa_R^{q+1}\right) \|s_k^H\|^{q+1}.
\end{align*}
From relation \eqref{den} we finally obtain:
\begin{equation*}
\lvert 1-\rho_k\rvert \leq 
\frac{(q+1)\left(L_H+L_h\kappa_R^{q+1}\right)}{q\lambda_k}.
\end{equation*}
Then, in both cases (when either a Taylor model or a lower level model is used), it exists a strictly positive constant $K$ such that the following relation holds:
\begin{align}\label{ratio2}
\lvert 1-\rho_k\rvert \leq \frac{K}{\lambda_k}, &&  K=\begin{cases} \displaystyle \frac{(q+1)L_h}{q}& \text{(Taylor model)}, \\
\displaystyle \frac{(q+1)\left(L_H+L_h\kappa_R^{q+1}\right)}{q} &\text{(lower level model)}.
\end{cases}
\end{align}

Using this last relation and the updating rule of the regularization parameter, we deduce that $\lambda_k$ must be bounded above. Indeed, in case of  unsuccessful iterations, $\lambda_k$ is increased. If $\lambda_k$ is increased, the ratio appearing in the right hand side of \eqref{ratio2} is progressively decreased, until it becomes smaller than $1-\eta_1$. In this case, $\rho_k>\eta_1$, so a successful step is taken and $\lambda_k$ is decreased. Hence $\lambda_k$ cannot be greater than
\begin{equation}\label{lambda_max}
\lambda_{\max}=\displaystyle \frac{K}{1-\eta_1}.
\end{equation} 

\subsubsection{Relating the steplength to the norm of the gradient}
Our next step is to show that the steplength cannot be arbitrarily small, compared to the norm of the gradient of the objective function.
If the Taylor model is used, from \cite[Lemma 2.3]{Birgin2017} it follows:
\begin{equation}\label{K1}
\|\nabla_x f^h(\xkh+s_k^h)\|\leq(L_h+\theta+\lambda_k)\|s_k^h\|^q:=K_1\|s_k^h\|^q.
\end{equation} 
If the lower level model is chosen, we have: 
\begin{align*}
\|R\nabla_x f^h(\xkh+s_k^h) \|\leq&\Big\|R\left[\nabla_x f^h(\xkh+s_k^h)-\nabla_s T_{q,k}^h(\xkh,s_k^h)\right]\Big\|\\
&+\|R\nabla_s T_{q,k}^h(\xkh,s_k^h)-\nabla_s m_{q,k}^H(x_{0,k}^H,s_k^H)\|\\
&+\|\nabla_s m_{q,k}^H (x_{0,k}^H,s_k^H)+\lambda_k\|s_k^H\|^{q-1}s_k^H\|+\lambda_k\|s_k^H\|^{q}.
\end{align*} 
By \eqref{taylor3}, the first term can be bounded by $\kappa_R L_h \|s_k^h\|^q$. Considering that $s_k^h=Ps_k^H$ and $\|P\|\leq \kappa_R$, we obtain the upper bound $\kappa_R^2 L_h \|s_k^H\|^q$. Regarding the second term, taking into account that from relations $s_k^h=Ps_k^H$, $R=P^T$, and \eqref{def2}, for all $p^H\in\mathbb{R}^{n_H}$ it holds:
\begin{align*}
\langle[\mathcal{R}(\nabla_x^if^h(x_k^h))](\underbrace{s_k^H,\dots,s_k^H}_{i-1\,\mathrm{times}}),p^H\rangle&=\langle\nabla_x^if^h(x_k^h,\underbrace{Ps_k^H,\dots,Ps_k^H}_{i-1\,\mathrm{times}}),Pp^H\rangle\\
&=\langle R[\nabla_x^if^h(x_k^h,\underbrace{Ps_k^H,\dots,Ps_k^H}_{i-1\,\mathrm{times}})],p^H\rangle,
\end{align*}
we can write 
\begin{align*}
R\nabla_s T_{q,k}^h(\xkh,Ps_k^H)&=\sum_{i=1}^{q}\frac{1}{(i-1)!}R\nabla^i_x f^h(\xkh)(\underbrace{Ps_k^H,\dots,Ps_k^H}_{i-1\,\mathrm{times}})\\
&=\sum_{i=1}^{q}\frac{1}{(i-1)!}\left[\mathcal{R}(\nabla^i_x f^h(\xkh))\right](\underbrace{s_k^H,\dots,s_k^H}_{i-1\,\mathrm{times}}).
\end{align*}

Then, from 
\begin{align}
\nabla_s m^H_{q,k}(x_{0,k}^H,s_k^H)=&\nabla_x f^H(x_{0,k}^H+s_k^H)\nonumber\\
&+\sum_{i=1}^{q}\frac{1}{(i-1)!}\left[\mathcal{R}(\nabla^i_x f^h(\xkh))-\nabla^i_x f^H(x_{0,k}^H)\right](\underbrace{s_k^H,\dots,s_k^H}_{i-1\,\mathrm{times}})\label{gradient_lower_model},
\end{align}
we obtain
\begin{align*}
\|R\nabla_s T_{q,k}^h(\xkh,s_k^h)-\nabla_s m_{q,k}^H(x_{0,k}^H,s_k^H)\|
&=\Big\|\nabla_x f^H(x_{0,k}^H+s_k^H)-\nabla_s T_{q,k}^H(x_{0,k}^H,s_k^H)\Big\|,
\end{align*}
which represents the Taylor remainder for the approximation of $\nabla_x f^H$ by  $\nabla_s T_{q,k}^H$.  Therefore, by relation \eqref{taylor3}, this quantity can be bounded above by 
$L_H\|s_k^H\|^{q}$.
The third term, from \eqref{stopping}, is less than $\theta\|s_k^H\|^{q}$. 
Then, since  $\lambda_k\leq\lambda_{\max}$, we finally obtain
\begin{equation}\label{K2}
\|R\nabla_x f^h(\xkh+s_k^h)\|\leq \left(\kappa_R^2L_h+L_H+\theta+\lambda_{\max}\right)\|s_k^H\|^{q}:=K_2\|s_k^H\|^{q}.
\end{equation}

\subsubsection{Proof of global convergence}
Let us consider the sequence of successful iterations ($\rho_k\geq\eta_1$). They are divided into two groups, $K_{s,f}$ the successful iterations at which the fine model has been employed and $K_{s,l}$ the ones at which the lower level model has been employed. 
Let us define $k_1$ the index of the first successful iteration. We remind that at successful iterations $\rho_k\geq\eta_1$. Due to the updating rule of the regularization parameter in Algorithm \ref{algo} we have $\lambda_k\geq\lambda_{\min}$. Hence from relations \eqref{pred}, \eqref{dent}, \eqref{K1} and \eqref{K2}, \eqref{go_down_condition} it follows that: 
\begin{align}
f^h(x_{k_1}^h)-&\liminf_{k\rightarrow\infty} f^h(\xkh) \geq\sum_{k succ}f^h(\xkh)-f^h(\xkh+s_k^h)\nonumber\\
\overset{\eqref{pred}}{\geq}&\eta_1 \sum_{K_{s,l}}(m_{q,k}^H(x_{0,k}^H)-m_{q,k}^H(x_{0,k}^H, s_k^H))+\eta_1 \sum_{K_{s,f}}(T_{q,k}^h(\xkh)-T_{q,k}^h(\xkh,s_k^h))\nonumber \\
\overset{\eqref{dent}}{\geq}& \frac{\eta_1\lambda_k}{q+1}\left(\sum_{K_{s,l}}\|s_k^H\|^{q+1}+\sum_{K_{s,f}}\|s_k^h\|^{q+1}\right)\nonumber\\
\overset{\eqref{K1}+\eqref{K2}}{\geq}& \frac{\eta_1\lambda_{\min}}{q+1}\Bigg(\frac{1}{K_2^{\frac{q+1}{q}}}\sum_{K_{s,l}}\|R\nabla_x f^h(\xkh+s_k^h)\|^{\frac{q+1}{q}}+\frac{1}{K_1^{\frac{q+1}{q}}}\sum_{K_{s,f}}\|\nabla_x f^h(\xkh+s_k^h)\|^{\frac{q+1}{q}}\Bigg)\nonumber \\
\overset{\eqref{go_down_condition}}{\geq}&
\frac{\eta_1\lambda_{\min}}{q+1}\Bigg(\frac{1}{K_2^{\frac{q+1}{q}}}\sum_{K_{s,l}}\kappa_H^{\frac{q+1}{q}}\|\nabla_x f^h(\xkh+s_k^h)\|^{\frac{q+1}{q}}+\frac{1}{K_1^{\frac{q+1}{q}}}\sum_{K_{s,f}}\|\nabla_x f^h(\xkh+s_k^h)\|^{\frac{q+1}{q}}\Bigg)\label{compl}.
\end{align}
Hence we conclude that $\sum_{K_{s,f}\cup K_{s,l}}\|\nabla_x f^h(\xkh+s_k^h)\|$ is a bounded series and therefore has a convergent subsequence. Then,  $\|\nabla_x f^h(\xkh+s_k^h)\|$ converges to zero on the subsequence of successful iterations. 

We can then state the global convergence property towards first-order critical points in the following theorem.
\begin{theorem}\label{teo_glob}
	Let Assumptions \ref{ass_R} and \ref{hp_lip_grad} hold. Let $\{\xkh\}$ be the sequence of fine level iterates generated by Algorithm \ref{algo}. Then, $\{\|\nabla_x f^h(\xkh)\|\}$ converges to zero on the subsequence of successful iterations. 
\end{theorem}

\subsection{Worst-case complexity}\label{sec_compl}
We now want to evaluate the worst-case complexity of our methods, to reach a first order stationary point.  We assume then that the procedure is stopped as soon as $\|\nabla_x f^h(\xkh)\|\leq\epsilon$ for $\epsilon>0$. The proof is similar to that of Theorem 2.5 in \cite{Birgin2017}. 

To evaluate the complexity of the proposed methods, we have to bound the number of successful and unsuccessful iterations performed before the stopping condition is met.  
Let us then define $k_f$ the index of the last  iterate for which $\|\nabla_x f^h(\xkh)\|>\epsilon$, $K_s=\{0<j\leq k_f\, |\, \rho_j\geq \eta_1\}$ the set of successful iterations before iteration $k_f$, and $K_u$ its complementary in $\{1,\dots,k_f\}$. 
We can use the same reasoning as that used to derive \eqref{compl}, but considering in the sum just the successful iterates in $K_s$. Remind that before termination  $\|\nabla_x f^h(\xkh)\|>\epsilon$ and, in case the lower level model is used, $\|R\nabla f^H(\xkh)\|>\kappa_H\|\nabla f^H(\xkh)\|>\kappa_H\epsilon$ (otherwise at that iteration the Taylor model would have been used). It then follows:
\begin{align*}
f^h(x_{k_1}^h)-\liminf_{k\rightarrow\infty} f^h(\xkh) &\geq f^h(x_{k_1}^h)- f^h(x_{k_f+1}^h)= \sum_{j\in K_s}f^h(\xkh)-f^h(\xkh+s_k^h)\\ &\geq\frac{\eta_1\lambda_{\min}}{q+1}\min\Big\{\frac{\kappa_H}{K_2},\frac{1}{K_1}\Big\}^{\frac{q+1}{q}}\lvert K_{s}\rvert\epsilon^{\frac{q+1}{q}},
\end{align*}
from which we get the desired bound on the total number of successful iterations. We can then bound the cardinality of $K_u$, with respect to the cardinality of $K_s$. From the updating rule of the regularization parameter, it holds:
\begin{equation*}
\gamma_1\lambda_k\leq\lambda_{k+1}, \, k\in K_s\,\qquad \gamma_3\lambda_k=\lambda_{k+1}, \, k\in K_u.
\end{equation*}
Then, proceeding inductively, we conclude that:
\begin{equation*}
\lambda_0\gamma_1^{\lvert K_s\rvert}\gamma_3^{\lvert K_u\rvert}\leq\lambda_{k_f}\leq \lambda_{\max}.
\end{equation*}
Then, 
\begin{equation*}
\lvert K_s\rvert \log\gamma_1+\lvert K_u\rvert \log\gamma_3\leq\log\frac{\lambda_{\max}}{\lambda_0}, 
\end{equation*}
and, given that $\gamma_1<1$, we obtain: 
\begin{equation*}
\lvert K_u\rvert \leq\frac{1}{\log\gamma_3}\log\frac{\lambda_{\max}}{\lambda_0}+\lvert K_s\rvert \frac{\lvert\log\gamma_1\rvert}{\log\gamma_3}.
\end{equation*}
We can then state the following result.
\begin{theorem}\label{teo_compl}
	Let Assumptions \ref{ass_R} and \ref{hp_lip_grad}. Let $f_{low}$ denote a lower bound on $f$ and let $k_1$ denote the index of the first successful iteration in Algorithm \ref{algo}. Then, given an absolute accuracy level $\epsilon >0$, Algorithm \ref{algo} needs at most 
	\begin{equation*}
	K_3 \frac{(f(x_{k_1})-f_{low})}{\epsilon^{\frac{q+1}{q}}}\Bigg(1+\frac{\lvert\log\gamma_1\rvert}{\log\gamma_3}\Bigg)+\frac{1}{\log\gamma_3}\log\left(\frac{\lambda_{\max}}{\lambda_0}\right)
	\end{equation*}
	iterations in total to produce an iterate $\xkh$ such that $\|\nabla_x f(x_k)\|\leq\epsilon$, where 
	\begin{equation*}
	K_3:=\frac{q+1}{\eta_1\lambda_{\min}}\max\Big\{K_1,\frac{K_2}{\kappa_H}\Big\}^{q+1/q},
	\end{equation*}
	with $K_1$ and $K_2$ defined in \eqref{K1}, \eqref{K2},  $\gamma_1,\gamma_3,\lambda_0,\lambda_{\min}$ defined in Algorithm \ref{algo} and $\lambda_{\max}$ defined in \eqref{lambda_max}.
\end{theorem}
Theorem \ref{teo_compl} reveals that the use of lower level steps does not deteriorate the complexity of the method, and that the complexity bound $O(\epsilon^{-\frac{q+1}{q}})$ is preserved. This is a very satisfactory result, because each iteration of the multilevel methods will be less expensive than one iteration of the corresponding one-level method, thanks to the use of the cheaper lower level models. Consequently, if the number of iterations in the multilevel strategy is not increased, we can expect global computational savings.

\subsection{Local convergence}\label{sec_loc}
In this section we study the local convergence of the proposed methods towards second-order stationary points. We assume $q\geq2$ in this section,  otherwise the problem is not well-defined. 
Thanks to the use of high order models, our methods are expected to attain a fast local convergence rate, especially for growing $q$. The results reported here are inspired by \cite{yue2018quadratic} and extend the analysis proposed therein.
 
We denote by $\mathcal{X}$ the set of second-order critical points of $f$, i.e. of points $x^*$ satisfying the second-order necessary conditions:
\begin{equation*}
\nabla_x f(x^*)=0,\quad \nabla_x^2 f(x^*)\succeq 0, 
\end{equation*}
i.e. $\nabla_x^2 f(x^*)$ is a symmetric positive semidefinite matrix. 
We denote by $\mathcal{B}(x, \rho)=\{y \,s.t.\, \|y-x\| \leq \rho\}$ and for all $x\in\mathbb{R}^n$,\, $\mathcal{L}(f(x))=\{y\in\mathbb{R}^n \,\vert\, f(y)\leq f(x)\}$ for $f: \mathbb{R}^n \rightarrow \mathbb{R}$.

\begin{remark}
	From the assumption that $f$ is a $q$ times continuously differentiable function, it follows that its $i$-th derivative tensor is locally Lipschitz continuous for all $i\leq q-1$. 
	\label{remark}
\end{remark}

\vspace{5pt}
 Following \cite{yue2018quadratic}, we first prove an intermediate lemma that allows us to relate, at generic iteration $k$, the norm of the step and the distance of the current iterate from the space of second-order stationary points. This lemma holds without need of assuming a stringent non-degeneracy condition, but rather under a local error bound condition, which is a much weaker requirement as it can be satisfied also when $f$ has non isolated second-order critical points. 
\begin{assumption}\label{ass_eb}
	There exist strictly positive scalars $\kappa_{EB},\rho>0$ such that
	\begin{equation}
	\mathrm{dist}(x,\mathcal{X})\leq\kappa_{EB}\|\nabla_x f(x)\|,\quad \forall x\in\mathcal{N}(\mathcal{X},\rho),
	\end{equation}
	where $\mathcal{X}$ is the set of second-order critical points of $f$, $\mathrm{dist}(x,\mathcal{X})$ denotes the  distance of $x$ to $\mathcal{X}$ and $\mathcal{N}(\mathcal{X},\rho)=\{x\; \vert \;\mathrm{dist}(x,\mathcal{X})\leq\rho\}$.
\end{assumption}
This condition has been proposed for the first time in \cite{yue2018quadratic}. It is different from other error bound conditions in the literature as, in contrast to them, $\mathcal{X}$ is not the set of first-order critical points, but of second-order-critical points. In addition to being  useful for proving convergence, it is also interesting on its own, as it is shown to be equivalent to a quadratic growth condition (\cite[Theorem 1]{yue2018quadratic}) under mild assumptions on  $f$.

\begin{lemma}\label{lemma}
	Let Assumptions \ref{ass_R} and \ref{hp_lip_grad} hold. Let $\{\xkh\}$ be the sequence generated by Algorithm \ref{algo} and $x_k^*$ be a projection point of $\xkh$ onto $\mathcal{X}$. Assume that it exists a strictly positive constant $\underline{\rho}$ such that $\{\xkh\}\in \mathcal{B}(x_k^*,\underline{\rho})$ and that $\nabla_x^2f$ is Lipschitz continuous in $\mathcal{B}(x_k^*,\underline{\rho})$ with Lipschitz constant $L_2$. Then, it holds:
	\begin{equation}\label{step_error}
	\|\skh\|\leq C \,\mathrm{dist}(\xkh,\mathcal{X}),
	\end{equation}
	with
	\begin{align*}
C=	\begin{cases}
	C_f= \displaystyle \frac{1}{2\lambda_{\max}}\left[L_2+\sqrt{L_2^2+4L_2\lambda_{\max}}\right],& \text{ (Taylor model)},\\
	\kappa_RC_c=\displaystyle \frac{\kappa_R}{2\lambda_{\max}}\left[\kappa_RL_2+\sqrt{\kappa_R^2L_2^2+4L_2\kappa_R\lambda_{\max}}\right], &\text{(lower level model)},
	\end{cases}
	\end{align*}
	with $\lambda_{\max}$ and $\kappa_R$ defined respectively in \eqref{lambda_max} and Assumption \ref{ass_R}.
	\end{lemma}

The proof of the lemma is reported in the appendix. This lemma can be used to prove that if it exists an accumulation point of $\{\xkh\}$ that belongs to $\mathcal{X}$, then the full sequence converges to that point and that the rate of convergence depends on $q$. First, we can prove that the set of accumulation points is not empty. 

\begin{lemma}\label{teo_glob2}
	Let Assumptions \ref{ass_R} and \ref{hp_lip_grad} hold. Let $\{\xkh\}$ be the sequence of fine level iterates generated by Algorithm \ref{algo}. If $\mathcal{L}(f(\xkh))$ is bounded for some $k\geq 0$, then the sequence has an accumulation point that is a  first-order stationary point. 
\end{lemma}
\begin{proof}
	As $\{f(\xkh)\}$ is a decreasing sequence, and $\mathcal{L}(f(\xkh))$ is bounded for some $k\geq 0$, $\{\xkh\}$ is a bounded sequence and it has an accumulation point. From Theorem \ref{teo_glob}, all the accumulation points are first-order stationary points. 
\end{proof}

\begin{theorem}\label{teo_loc}
	Let Assumptions \ref{ass_R} and \ref{hp_lip_grad} hold. Let $\{x_k^h\}$ be the sequence of fine level iterates generated by Algorithm \ref{algo}. Assume that $\mathcal{L}(f(x_k^h))$ is bounded for some $k\geq 0$   and that   it exists an accumulation point $x^*$ such that $x^*\in\mathcal{X}$. Then, the whole sequence $\{x_k^h\}$ converges to $x^*$ and it exist strictly positive constants $c \in \mathbb{R}$ and $\bar{k} \in \mathbb{N}$ such that:
	\begin{equation}\label{rate}
	\frac{\|x_{k+1}^h-x^*\|}{\|x_k^h-x^*\|^q}\leq c, \quad \forall k\geq\bar{k}.
	\end{equation}
\end{theorem}
\begin{proof}
	As $x^*$ is an accumulation point, we have that $\lim\limits_{k\rightarrow\infty}\mathrm{dist}(x_k^h,\mathcal{X})=0$. Then, it exist $\rho$ and $k_1$ such that $x_k^h\in \mathcal{N}(\mathcal{X},\rho)$ for all $k\geq k_1$. Therefore, from Assumption \ref{ass_eb} it holds
	\begin{equation}\label{rel_eb}
	\mathrm{dist}(x_k^h,\mathcal{X})\leq \kappa_{EB} \, \|\nabla_x f^h(x_k^h)\|, \quad \forall k\geq k_1.
	\end{equation}
	Moreover, from Remark \ref{remark}, $\nabla_x^2 f$ is locally Lipschitz continuous, so Lemma \ref{lemma} applies to all $k\geq k_1$. 
	
	Let us first consider the case in which the Taylor model is employed. 
	It follows from \eqref{rel_eb}, \eqref{K1} and \eqref{step_error} that for all $k\geq k_1$
	\begin{align*}
	\mathrm{dist}(x_{k+1}^h,\mathcal{X})\leq& \kappa_{EB} \, \|\nabla_x f^h(x_{k+1}^h)\|\leq \kappa_{EB} \, K_1 \, \|s_k^h\|^q\leq \kappa_{EB} \, K_1 \, C_f^q \, \mathrm{dist}^q(x_k^h,\mathcal{X}).
	\end{align*} 
	If the lower level model is employed,  from \eqref{rel_eb}, \eqref{go_down_condition}, \eqref{K2} and \eqref{step_error_coarse} it follows that for all $k\geq k_1$
	\begin{align*}
	\mathrm{dist}(x_{k+1}^h,\mathcal{X})\leq& \kappa_{EB} \, \|\nabla_x f^h(x_{k+1}^h)\| \leq  \kappa_{EB} \, \kappa_H \, \|R\nabla_x f^h(x_{k+1}^h)\|\\
	\leq & \kappa_{EB} \, \kappa_H \, K_2 \, \|s_k^H\|^q\leq  \kappa_{EB} \, \kappa_H \, K_2 \,  C_c^q \, \mathrm{dist}^q(x_k^h,\mathcal{X}).
	\end{align*}
	Then in both cases, it exists $\bar{C}$ such that 
	\begin{equation*}
	\mathrm{dist}(x_{k+1}^h,\mathcal{X})\leq \bar{C} \, \mathrm{dist}^q(x_k^h,\mathcal{X}), \quad \forall k\geq k_1,
	\end{equation*}
	where
	\begin{equation*}
	 \bar{C}=\begin{cases}\kappa_{EB} K_1 \, C_f^q & \text{(Taylor model)},\\
	 \kappa_{EB}\kappa_H \, K_2 \, C_c^q &\text{(lower level model)}.
	\end{cases}
	\end{equation*}
	With this result, we can prove the convergence of $\{x_k^h\}$ with standard arguments. We repeat for example the arguments of the proof of \cite[Theorem2]{yue2018quadratic} for convenience.  Let $\eta>0$ be an arbitrary value. As $\lim_{k\rightarrow\infty}\mathrm{dist}(x_k^h,\mathcal{X})=0$, it exists $k_2\geq 0$ such that
	\begin{equation*}
	\mathrm{dist}(x_k^h,\mathcal{X})\leq\min\Big\{\frac{1}{2\bar{C}},\frac{\eta}{2C}\Big\}, \quad \forall k\geq k_2.
	\end{equation*}
	Then,
	\begin{equation*}
	\mathrm{dist}(x_{k+1}^h,\mathcal{X})\leq\bar{C}\mathrm{dist}^q(x_{k}^h,\mathcal{X})\leq \frac{1}{2}\mathrm{dist}(x_{k}^h,\mathcal{X}),\quad \forall k\geq \bar{k}=\max\{k_1,k_2\}.
	\end{equation*}
	From \eqref{step_error}, it then holds for all $k\geq \bar{k}$ and $j\geq 0$:
	\begin{align*}
	\|x_{k+j}^h-x_k^h\|&\leq\sum_{i=k}^{\infty}\|x_{i+1}^h-x_i^h\|\leq \sum_{i=k}^{\infty}C\mathrm{dist}(x_i^h,\mathcal{X})\\&\leq C\mathrm{dist}(x_k^h,\mathcal{X})\sum_{i=0}^{\infty}\frac{1}{2^i}
	\leq 2C\mathrm{dist}(\xkh,\mathcal{X})\leq\eta,
	\end{align*}
	i.e. that $\{x_k^h\}_{k\geq \bar{k}}$ is a Cauchy sequence and so the whole sequence is convergent.
	Finally we establish the $q$-th order rate of convergence of the sequence. For any $k\geq \bar{k}$, 
	\begin{equation}
	\|x^*-x_{k+1}^h\|=\lim_{j\rightarrow\infty}\|x_{k+j+1}^h-x_{k+1}^h\|\leq 2C\mathrm{dist}(x_{k+1}^h,\mathcal{X})\leq 2C\bar{C}\mathrm{dist}^q(x_k^h,\mathcal{X}).
	\end{equation}
	Combining this with $\mathrm{dist}(x_k^h,\mathcal{X})\leq\|x_k^h-x^*\|$, and setting $c=2C\bar{C}$ we obtain the thesis \eqref{rate}.
	
	Therefore $\{x_k^h\}$ converges at least with order $q$ to $x^*$.
	\end{proof}

	\section{Numerical results}\label{sec_numerical_results}
In this section, we report on the practical performance of a method in the family. 
	 
	We have implemented the method corresponding to $q=2$ in Algorithm \ref{algo} in Julia \cite{beks:17} (version 0.6.1). This is a multilevel extension of the method AR$2$ in Algorithm \ref{algo0}, which is better known as  ARC \cite{arc,arc2}. We will therefore denote the implemented multilevel method as MARC (multilevel adaptive method based on cubic regularization), rather than MAR$2$. 
	
	We consider the following two-dimensional nonlinear problem in the unit square domain $S_2$:
	\begin{equation*}
	\begin{cases}
	-\Delta u(x, y)+e^{u(x,y)} = g(x,y) & \text{in}\; S_2 ,\\
	u(x, y) = 0\quad &\text{on} \;\partial S_2,
	\end{cases}
	\end{equation*}	
	where $g$ is obtained such that the analytical solution to this problem is given by
$$	u(x, y) = \sin(2\pi x(1 - x)) \sin(2\pi y(1-y)).$$
	The negative Laplacian operator is discretized using finite difference, giving a symmetric positive definite matrix $A$,  that also takes into account the boundary conditions.
	The discretized version of the problem is then a system of the form $Au+e^u=g$, where $u$, $g$, $e^u$ are vectors in $\mathbb{R}^{n_h}$, in which the columns of matrices $U$, $G$, $E$ are stacked, with $U_{i,j}=u(x_i,y_j)$, $G_{i,j}=g(x_i,y_j)$, $E_{i,j}=e^{u(x_i,y_j)}$, for $x_i,y_j$ grid points, $i,j=1,\dots,\sqrt{n_h}$.
	
	 The MARC algorithm is then used on the nonlinear minimization problem
	\begin{equation}\label{pb_opti}
	\min_{u\in \mathbb{R}^{n_h}} \frac{1}{2} u^TAu+\|e^{u/2}\|^2-g^Tu,
	\end{equation}
	which is equivalent to the system $Au+e^u=g$. The coarse approximations to the objective function arise from a coarser discretization of the problem. Each coarse two-dimensional grid has a dimension that is four times lower than the dimension of the grid on the corresponding upper level.
	
		The prolongation operators $P_\ell$ from level $\ell-1$ to $\ell$ are based on the nine-point interpolation scheme defined by the stencil
		$
		\begin{pmatrix}
		\frac{1}{4}\; \frac{1}{2}\; \frac{1}{4}\\
		\frac{1}{2} \;1 \;\frac{1}{2}\\
		\frac{1}{4}\; \frac{1}{2} \;\frac{1}{4}
		\end{pmatrix}
		$ 
		and the full weighting operators defined by $R_\ell=\frac{1}{4}P_\ell^T$ are used as restriction operators \cite{briggs}.
	
 We compare the one-level ARC with MARC. Parameters common to both methods are set as: $\epsilon^{l_{\max}}=10^{-7}$, $\gamma_1=0.85$, $\gamma_2=0.5$, $\gamma_3=2$, $\lambda_0=0.05$ $\eta_1=0.1$, $\eta_2=0.75$. For MARC we set  $\kappa=0.1$ and $\epsilon^\ell=\epsilon^{\ell_{\max}}$ for all $\ell$. 
 
 At each iteration we find an approximate minimizer of the cubic models as described in \cite[\S 6.2]{arc}. This requires a sequence of Cholesky factorizations, which represents the dominant cost per nonlinear iteration. We measure the performance of the methods in terms of total number of floating point iterations required for these factorizations. 

 We study the effect of the multilevel strategy on the convergence of the method for problems of fixed dimension $n_h$.
  We then consider the solution of problem \eqref{pb_opti} using two different discretizations with $n_h=4096$ (Table \ref{tab_1}) and $n_h=16384$ (Table \ref{tab_2}), respectively. We allow $4$ levels in MARC.
 We report the results of the average of ten simulations with different random initial guesses of the form $u_0=a\,\mathrm{rand}(n_h,1)$,  for different values of $a$.  In each simulation the random starting guess is the same for the two considered methods. All the quantities reported in Tables \ref{tab_1} and \ref{tab_2} are the average of the values obtained over the ten simulations. $it_T$ denotes the number of total iterations, $it_f$ denotes the number of iterations in which the Taylor  model has been used, ${\tt RMSE}$ is the root-mean square error with respect to the true solution and {\tt save} is the ratio between the total number of floating point operations required for the Cholesky factorizations by ARC and MARC, respectively. For the {\tt save} quantity, we report three values: the minimum, the average and the maximum value obtained over the ten simulations. 
 
 The results reported in Tables \ref{tab_1} and \ref{tab_2} confirm the relevance of MARC as compared to ARC. The numerical experiments highlight the different convergence properties of both algorithms. The use of MARC is especially convenient when the initial guess is not so close to the true solution. Indeed,  the performance of ARC deteriorates as the distance of the initial guess from the  true solution increases,  while MARC seems to be much less sensible to this choice. For the problem of smaller dimension, ARC still manages to find a solution for further initial guesses, even if this requires an higher number of iterations, while for the problem of larger dimension the method fails to find a solution in feasible time. The new multilevel approach is found to lead to considerable computational savings in terms of floating point operations compared to the classical one-level strategy.

  \begin{table}
  	\centering
  	\begin{tabular}{c|c|ccc}
  	&	 &&$n_h=4096$ &\\
  $u_0$	&	Method 	 & $it_T/it_f$& {\tt RMSE} &{\tt save}\\
  		\hline
  		$\bar{u}_1$ 	&	ARC	&6/6  & $10^{-4}$ & \\  
  		&	MARC	& 9/4& $10^{-4}$ &1.7-2.0-2.3 \\
  		\hline
$\bar{u}_2$  	&	ARC	&17/17 & $10^{-4}$ &\\  
  	&	MARC	& 10/3& $10^{-4}$ &1.9-5.8-8.3 \\
  	\end{tabular}
  	\caption{Solution of the minimization problem (\ref{pb_opti}) with the one level ARC method and a four level ARC (MARC) (case of $n_h=4096$) with $\bar{u}_1=1\,\mathrm{rand}(n_h,1)$, $\bar{u}_2=3\,\mathrm{rand}(n_h,1)$. $it_T$ denotes the average number of iterations over ten simulations, $it_f$ the average number of iterations in which the fine level model has been used, ${\tt RMSE}$ the root-mean square error with respect to the true solution  and {\tt save} the ratio between the total number of floating point operations required for the Cholesky factorizations in ARC and MARC, respectively.}
  	\label{tab_1}
  \end{table}
  
  \begin{table}
  	\centering
  	\begin{tabular}{c|c|ccc}
  		&	 & &$n_h=16384$&\\
  		$u_0$	&	Method 	 & $it_T/it_f$& {\tt RMSE} &{\tt save}  \\ 
  		\hline
  		$\bar{u}_1$ 	&	ARC	& 6/6& $10^{-5}$& \\  
  		&	MARC	&12/3&$10^{-5}$ & 1.5-2.0-2.5\\
  		\hline
  		$\bar{u}_3$  	&	ARC	  &FAIL  & FAIL& \\  
  		&	MARC	& 18/5&$10^{-5}$ & -\\
  	\end{tabular}
  	\caption{Solution of the minimization problem (\ref{pb_opti}) with the one level ARC method and a four level ARC (MARC) (case of $n_h=16384$) with $\bar{u}_1=1\,\mathrm{rand}(n_h,1)$, $\bar{u}_3=6\,\mathrm{rand}(n_h,1)$. $it_T$ denotes the average number of iterations over ten simulations, $it_f$ the average number of iterations in which the fine level model has been used, ${\tt RMSE}$ the root-mean square error with respect to the true solution  and {\tt save} the ratio between the total number of floating point operations required for the Cholesky factorizations in ARC and MARC, respectively.}
  	\label{tab_2}
  \end{table}


	\section{Conclusions}\label{sec_conclusion}
	
	We have introduced a family of multilevel methods of order $q\geq 1$ for unconstrained minimization. These methods represent an extension of the higher-order methods presented in \cite{Birgin2017} and of the multilevel trust-region method proposed in \cite{rmtr}. We have proposed a unifying framework to analyse these methods, which is useful to prove their convergence properties and evaluate their worst-case complexity to reach first-order stationary points. As expected, we show that the local rate of convergence and the complexity bound depend on $q$ and high values of $q$ allow both fast local convergence and lower complexity bounds.
	
	 We believe this represents a contribution in the optimization field, as the use of multilevel ideas allows to reduce the major cost per iteration of the high-order methods. This gives a first answer to  the question posed in \cite{Birgin2017} about whether the approach presented there can have practical implications, in applications for which computing $q$ derivatives is feasible.
	 
	  We have implemented the multilevel method corresponding to $q=2$ and presented numerical results that show the considerable benefits of the multilevel strategy in terms of savings in floating point operations. Additional numerical results can be found in \cite{multilevel_ann}, where the authors apply the multilevel method in the family corresponding to $q=1$ to problems arising in the training of artificial neural networks for the approximate solution of partial differential equations. This case is particularly interesting as it allows to show the efficiency of multilevel methods even for problems without an underlying geometrical structure.

	
	%
\bibliographystyle{plain}
\bibliography{biblio_paper_arxiv}

\appendix
\section{Proof of Lemma 4.5 } 
In this appendix we report the proof of Lemma \ref{lemma}. We restate it here for convenience of the reader.

\begin{lemma*}[Lemma 4.5]
	Let Assumptions \ref{ass_R} and \ref{hp_lip_grad} hold. Let $\{\xkh\}$ be the sequence generated by Algorithm \ref{algo} and $x_k^*$ be a projection point of $\xkh$ onto $\mathcal{X}$. Assume that it exists a strictly positive constant $\underline{\rho}$ such that $\{\xkh\}\in \mathcal{B}(x_k^*,\underline{\rho})$ and that $\nabla_x^2f$ is Lipschitz continuous in $\mathcal{B}(x_k^*,\underline{\rho})$ with Lipschitz constant $L_2$. Then, it holds:
	\begin{equation*}
	\|\skh\|\leq C \,\mathrm{dist}(\xkh,\mathcal{X}),
	\end{equation*}
	with
	\begin{align*}
	C=	\begin{cases}
	C_f= \displaystyle \frac{1}{2\lambda_{\max}}\left[L_2+\sqrt{L_2^2+4L_2\lambda_{\max}}\right],& \text{ (Taylor model)},\\
	\kappa_RC_c=\displaystyle \frac{\kappa_R}{2\lambda_{\max}}\left[\kappa_RL_2+\sqrt{\kappa_R^2L_2^2+4L_2\kappa_R\lambda_{\max}}\right], &\text{(lower level model)},
	\end{cases}
	\end{align*}
	with $\lambda_{\max}$ and $\kappa_R$ defined respectively in \eqref{lambda_max} and Assumption \ref{ass_R}.
\end{lemma*}
	\begin{proof}
		The proof is divided into two parts. We first consider the case in which $\skh$ has been obtained from the approximate minimization of the Taylor model, and then the case in which it has been obtained as prolongation of the step obtained from the approximate minimization of the coarse model. 
		
		Le us then assume that the Taylor model has been employed. Reminding that $\nabla_x f^h(x^*)=0$ for each $x^*\in\mathcal{X}$, and definition \eqref{taylor_model} we obtain:
		\begin{align}
		\nabla_s \left(m_{q,k}^h(\xkh,s_k^h)+\frac{\lambda_k}{q+1}\|\skh\|^{q+1}\right)=&-\nabla_x f^h(x_k^*)+\nabla_x f^h(\xkh)+\nabla_x^2 f^h(\xkh)s_k^h\nonumber\\&+H(s_k^h)+\lambda_k\|s_k^h\|^{q-1}s_k^h\label{base},
		\end{align}
		with 
		\begin{equation*}
		H(s_k^h)=\sum_{i=3}^{q}\frac{1}{(i-1)!}\nabla^i_x f^h(\xkh,\underbrace{\skh,\dots,\skh}_{i-1\,\mathrm{times}}).
		\end{equation*}
		Some algebraic manipulations (adding $\nabla_x^2 f^h(x_k^*)(x_{k+1}^h-x_k^*)$ to both sides of \eqref{base} and expressing $(x_{k+1}^h-x_k^*)=s_k^h+(x_k^h-x_k^*)$) lead to:
		\begin{align*}
		&\left(\nabla_x^2 f^h(x_k^*)+\lambda_k\|s_k^h\|^{q-1}\right)(x_{k+1}^h-x_k^*)=\\
		&\nabla_s \left(m_{q,k}^h(\xkh, \skh)+\frac{\lambda_k}{q+1}\|s_k^h\|^{q+1}\right)+\nabla_x f^h(x_k^*)-\nabla_x f^h(\xkh)-\nabla_x^2 f^h(x_k^*)(x_k^*-\xkh)\\
		&-H(s_k^h)+(\nabla_x^2 f^h(x_k^*)-\nabla_x^2 f^h(\xkh))s_k^h-\lambda_k\|s_k^h\|^{q-1}(x_k^*-\xkh).
		\end{align*}
		Using the fact that $\nabla_x^2 f^h(x_k^*)\succeq 0$, the stopping criterion \eqref{stopping}, and the triangle inequality, it follows
		\begin{align*}
		\lambda_k\|s_k^h\|^{q-1}\|x_{k+1}^h-x_k^*\|\leq & \, \theta\|s_k^h\|^q+
		\|\nabla_x f^h(x_k^*)-\nabla_x f^h(\xkh)-\nabla_x^2 f^h(x_k^*)(x_k^*-\xkh)\|+\\
		&\|H(s_k^h)\|+\|\nabla_x^2 f^h(x_k^*)-\nabla_x^2 f^h(\xkh)\|\|s_k^h\|+\lambda_k\|s_k^h\|^{q-1}\|x_k^*-\xkh\|.
		\end{align*}
		Using the Lipschitz continuity of $\nabla_x^2 f$ in $\mathcal{B}(x_k^*,\underline{\rho})$, the relation \eqref{taylor3} with $q=2$ and the triangle inequality $\|x_{k+1}^h-x_k^*\|\geq\|x_{k+1}^h-\xkh\|-\|\xkh-x_k^*\|=\|\skh\|-\|\xkh-x_k^*\|$, we obtain:
		\begin{align*}
		\lambda_k\|s_k^h\|^q\leq& \, \theta\|\skh\|^q+
		L_2\|\xkh-x_k^*\|^2+\|H(s_k^h)\|+L_2\|x_k^*-\xkh\|\|s_k^h\|+2\lambda_k\|s_k^h\|^{q-1}\|x_k^*-\xkh\|.
		\end{align*}
		Notice that 
		\begin{align*}
		&
		\theta\|\skh\|^q+L_2\|\xkh-x_k^*\|^2+\|H(s_k^h)\|+L_2\|x_k^*-\xkh\|\|s_k^h\|+2\lambda_k\|s_k^h\|^{q-1}\|x_k^*-\xkh\|\\
		&\geq
		L_2\|\xkh-x_k^*\|^2+L_2\|x_k^*-\xkh\|\|s_k^h\|.
		\end{align*}
		We can then study when the inequality holds
		\begin{equation*}
		\lambda_k\|\skh\|^q\leq L_2\|\xkh-x_k^*\|^2+L_2\|x_k^*-\xkh\|\|s_k^h\|.
		\end{equation*}
		The right hand side of the inequality is expressed as a polynomial of $\|s_k^h\|$ of order $1$ with positive value in 0, so the inequality will be true if $\|s_k^h\|$ is small enough. We can then assume $\|s_k^h\|<1$, so that $\|s_k^h\|^q\leq\|s_k^h\|^2$ if $q\geq2$. Then, we have that 
		\begin{equation*}
		\lambda_k\|s_k^h\|^q\leq\lambda_k\|s_k^h\|^2.
		\end{equation*}
		We can then solve
		\begin{align*}
		L_2\|\xkh-x_k^*\|^2+L_2\|x_k^*-\xkh\|\|s_k^h\|-\lambda_k\|\skh\|^2&\geq 0.
		\end{align*}
		The solution leads to
		\begin{equation}\label{step_error_fine}
		\|\skh\|\leq C_f\|\xkh-x_k^*\| , \quad
		C_f=\frac{1}{2\lambda_k}\left[L_2+\sqrt{L_2^2+4L_2\lambda_k}\right].
		\end{equation}

		Let us now consider the case in which the lower level model is used. The idea is similar as in the previous case. Reminding \eqref{gradient_lower_model} and that $R\nabla_x f^h(x_k^*)=0$, we have:
		\begin{align*}
		\nabla_s \left(m^H_k(x_{0,k}^H,s_k^H)+\frac{\lambda_k}{q+1}\|s_k^H\|^{q+1}\right)=&\nabla_x f^H(x_{0,k}^H+s_k^H)-\nabla_s T_{q,k}^H(x_{0,k}^H,s_k^H)+R\nabla_x f^h(x_k^*)+\\
		&\sum_{i=1}^{q}\frac{1}{(i-1)!}\mathcal{R}(\nabla^i_x f^h(x_k^h))\underbrace{(s_k^H,\dots,s_k^H)}_{i-1\,\mathrm{times}}+\lambda_k\|s_k^H\|^{q-1}s_k^H.
		\end{align*}
		Algebraic manipulations (adding $R\nabla_x^2 f^h(x_k^*)(x_{k+1}^h-x_k^*)$ to both sides and expressing $(x_{k+1}^h-x_k^*)=s_k^h+(x_{k}^h-x_k^*)$) lead to:
		\begin{align*}
		R\nabla_x^2 f^h(x_k^*)&(x_{k+1}^h-x_k^*)=\nabla_s\left(m^H_k(x_{0,k}^H,s_k^H)+\frac{\lambda_k}{q+1}\|s_k^H\|^{q+1}\right)-\nabla_x f^H(x_{0,k}^H+s_k^H)\\&+\nabla_s T_{q,k}^H(x_{0,k}^H, s_k^H)+R\nabla_x f^h(x_k^*)-R\nabla_x f^h(\xkh)-R\nabla_x^2 f^h(x_k^*)(x_k^*-\xkh)\\
		&-H_H(s_k^H)+R(\nabla_x^2 f^h(x_k^*)-\nabla_x^2 f^h(\xkh))s_k^h-\lambda_k\|s_k^H\|^{q-1}s_k^H,
		\end{align*}
		where
		\begin{equation*}
		H_H(s_k^H)=\sum_{i=3}^{q}\frac{1}{(i-1)!}\left[\mathcal{R}(\nabla^i_x f^h(x_k^H))\right](\underbrace{s_k^H,\dots,s_k^H}_{i-1\,\mathrm{times}}).
		\end{equation*}
		Further, we can write $R\nabla_x^2 f^h(x_k^*)(x_{k+1}^h-x_k^*)=R\nabla_x^2 f^h(x_k^*)(x_{k+1}^h-x_k^h)+R\nabla_x^2 f^h(x_k^*)(x_{k}^h-x_k^*)=R\nabla_x^2 f^h(x_k^*)Ps_k^H+R\nabla_x^2 f^h(x_k^*)(x_{k}^h-x_k^*)$:
		\begin{align*}
		(R\nabla_x^2 f^h(x_k^*)P+&\lambda_k\|s_k^H\|^{q-1})s_k^H=\nabla_s\left( m_{q,k}^H(x_k^H, s_k^H)+\frac{\lambda_k}{q+1}\|s_k^H\|^{q}\right)-\nabla_x f^H(x_{0,k}^H+s_k^H)\\&+\nabla_s T_{q,k}^H(x_{0,k}^H, s_k^H)
		+R\nabla_x f^h(x_k^*)-R\nabla_x f^h(\xkh)-R\nabla_x^2 f^h(x_k^*)(x_k^*-\xkh)\\
		&-H_H(s_k^H)+R(\nabla_x^2 f^h(x_k^*)-\nabla_x^2 f^h(\xkh))Ps_k^H-R\nabla_x^2f^h(x_k^*)(\xkh-x_k^*).
		\end{align*}
		We can again use relation \eqref{taylor3} (applied to $f^H,T_{q,k}^H$ with constant $L_H$ and to $f,T_{2,k}^h$ with constant $L_2$), \eqref{stopping}, the fact that $R\nabla_x^2 f^h(x_k^*)P$ is still positive definite, and Assumption \ref{ass_R} together with relation $s_k^h=Ps_k^H$, to deduce that:
		\begin{align*}
		\lambda_k\|s_k^H\|^{q}&\leq (\theta+L_H)\|s_k^H\|^q+\kappa_RL_2\|x_k^*-\xkh\|^2
		+\|H_H(s_k^H)\|\\&+\kappa_R^2L_{2}\|x_k^*-\xkh\|\|s_k^H\|+\|R\nabla_x^2f^h(x_k^*)(\xkh-x_k^*)\|.
		\end{align*}
		We remark that 
		\begin{align*}
		(\theta+L_H)\|s_k^H\|^q+\kappa_RL_2\|x_k^*-\xkh\|^2
		+\|H_H(s_k^H)\|&+\kappa_R^2L_{2}\|x_k^*-\xkh\|\|s_k^H\|\\+\|R\nabla_x^2f^h(x_k^*)(\xkh-x_k^*)\|&\geq \kappa_RL_2\|x_k^*-\xkh\|^2+\kappa_RL_{2}\|x_k^*-\xkh\|\|s_k^H\|.
		\end{align*}
		As previously, we can solve the following inequality:
		\begin{align*}
		\lambda_k\|s_k^H\|^{2}&\leq \kappa_RL_2\|x_k^*-\xkh\|^2+\kappa_R^2L_{2}\|x_k^*-\xkh\|\|s_k^H\|,
		\end{align*}
		and conclude that:
		\begin{equation}\label{step_error_coarse}
		\|s_k^H\|\leq C_c \|\xkh-x_k^*\|, \quad C_c=\frac{\left[\kappa_RL_2+\sqrt{\kappa_R^2L_2^2+4L_2\kappa_R\lambda_k}\right]}{2\lambda_k}.
		\end{equation}
		We can then use the fact that $\lambda_k\leq\lambda_{\max}$ for all $k$ and that $\|\skh\|\leq\kappa_R\|s_k^H\|$  to conclude that in all cases it exists a constant $C$ such that $\|\skh\|\leq C \|\xkh-x_k^*\|$.
		\end{proof}

\end{document}